\documentclass[a4paper,11pt]{article}
\usepackage{amsmath, amsfonts, amscd, amssymb, amsthm, enumerate, xypic}

\DeclareMathOperator{\id}{\operatorname{id}}

\DeclareMathOperator{\Mat}{\operatorname{M}}

\DeclareMathOperator{\End}{\operatorname{End}}
\DeclareMathOperator{\Diag}{\operatorname{Diag}}

\DeclareMathOperator{\Ker}{\operatorname{Ker}}

\DeclareMathOperator{\Vect}{\operatorname{span}}
\DeclareMathOperator{\im}{\operatorname{Im}}

\DeclareMathOperator{\tr}{\operatorname{tr}}
\DeclareMathOperator{\Tr}{\operatorname{Tr}}
\DeclareMathOperator{\str}{\operatorname{str}}
\DeclareMathOperator{\car}{\operatorname{char}}

\DeclareMathOperator{\rk}{\operatorname{rk}}

\renewcommand{\setminus}{\smallsetminus}
\renewcommand{\epsilon}{\varepsilon}

\def\F{\mathbb{F}}

\def\calF{\mathcal{F}}

\def\calH{\mathcal{H}}

\def\calQ{\mathcal{Q}}

\def\calW{\mathcal{W}}

\def\lcro{\mathopen{[\![}}
\def\rcro{\mathclose{]\!]}}

\theoremstyle{definition}

\theoremstyle{plain}
\newtheorem{theo}{Theorem}[section]
\newtheorem{prop}[theo]{Proposition}
\newtheorem{cor}[theo]{Corollary}
\newtheorem{lemma}[theo]{Lemma}
\newtheorem{claim}{Claim}

\theoremstyle{plain}

\theoremstyle{remark}
\newtheorem{Rems}{Remarks}[section]
\newtheorem{Rem}[Rems]{Remark}

\title{Sums of two nilpotent quaternionic matrices}
\author{Cl\'ement de Seguins Pazzis\footnote{Universit\'e de Versailles Saint-Quentin-en-Yvelines, Laboratoire de Math\'ematiques
de Versailles, 45 avenue des \'Etats-Unis, 78035 Versailles cedex, France}
\footnote{e-mail address: clement.de-seguins-pazzis@ac-versailles.fr}}

\begin{document}

\thispagestyle{plain}

\maketitle
\begin{abstract}
Let $\calQ$ be a quaternion division algebra over a field $\F$, and $n \geq 2$ be an integer.
In a recent article, de la Cruz et al have proved that every $n$-by-$n$ matrix with entries in $\calQ$ and pure quaternionic trace
is the sum of three nilpotent matrices, and they have shown that some are not the sum of two nilpotent matrices.

Here, we give a simple characterization of the square matrices with entries in $\calQ$ that are the sum of two nilpotent ones.
When $n \geq 3$, the special cases involve the scalar matrices and their perturbations by rank $1$ matrices, as well as the very special
case of $3$-by-$3$ unispectral diagonalisable matrices.
\end{abstract}



\section{Introduction}

\subsection{The problem}

The starting point of this article is the following classical theorem:

\begin{theo}\label{theo:basic}
Let $\F$ be a field, and $A \in \Mat_n(\F)$ be a nonscalar matrix (i.e. $A \neq \lambda I_n$ for all $\lambda \in \F$).
Then the following conditions are equivalent:
\begin{enumerate}[(i)]
\item $A$ has trace zero;
\item $A$ is similar to a matrix with diagonal zero;
\item $A$ is the sum of two nilpotent matrices.
\end{enumerate}
\end{theo}

In this theorem, the implication (iii) $\Rightarrow$ (i) is obvious because the trace is linear and
the trace of any nilpotent matrix is zero. The implication (ii) $\Rightarrow$ (iii) is easy, as we can split any matrix with diagonal
zero into the sum of a strictly upper-triangular matrix and a strictly lower-triangular one, both of which are nilpotent.
The only nontrivial part is (i) $\Rightarrow$ (ii), which has been known for a long time
(see e.g.\ \cite{Halmos} p.109 exercise 1) and is a special case of a theorem of Fillmore \cite{Fillmore} on the possible diagonal vectors of the matrices
in the similarity class of a matrix.

Here, we are interested in analogues of Theorem \ref{theo:basic} over (noncommutative) division rings.
To set things straight immediately, for us a division ring $D$ (also called a skew field) is an associative ring with unity $1_D \neq 0_D$
in which every nonzero element is invertible. It should also be pointed out that Theorem \ref{theo:basic} fails over general division rings,
even if we use the correct definition of a scalar matrix in that case (that is, a diagonal matrix with all diagonal entries equal \emph{and in the center} of $D$,
or equivalently a central element of the ring $\Mat_n(D)$).
One of the problems lies in the fact that the trace of a matrix, defined as the sum of all its diagonal entries, is no longer invariant
under similarity. In fact, given two $n$-by-$n$ matrices $A,B$ with entries in a ring $R$, the difference $\Tr(AB)-\Tr(BA)$ lies
in the additive subgroup $\calW$ of $(R,+)$ generated by the Lie commutators $[a,b]:=ab-ba$ with $a,b$ in $R$, and the correct invariant is then the
coset of $\Tr(A)$ in the quotient group $(R/\calW,+)$, called the \textbf{reduced trace} and denoted by $\tr(A)$ (with small ``t").
The reduced trace of any nilpotent matrix over $D$ is zero (because every such matrix is similar to a strictly upper-triangular one),
and hence the reduced trace of any sum of nilpotent matrices is zero.
Conversely, it is known that every matrix with entries in a division ring $D$ and
whose trace is a sum of Lie commutators is a sum of nilpotent matrices \cite{Harris}, but bounding the number of summands in such a decomposition depends on the division ring under consideration
(see, e.g., the last part of \cite{Breaz} for some results on the matter).

The present article is motivated by the recent article of de la Cruz et al \cite{dLC}, in which the problem is considered over the quaternion division rings.
These division rings are naturally interesting because they are the noncommutative ones that are closest to commutativity in terms of degree over their center.
To start with, let $\F$ be a field (possibly of characteristic $2$). A quaternion algebra over $\F$ can be defined in various equivalent ways:
\begin{itemize}
\item As an $\F$-algebra that is isomorphic to the Clifford algebra of a regular $2$-dimensional quadratic form over $\F$;
\item As a $4$-dimensional central (i.e., the center is reduced to the $\F$-scalar multiples of the unity)
and simple (i.e., without proper two-sided ideals) algebra over $\F$.
\end{itemize}
Let $\calQ$ be such a quaternion algebra. There is a unique involution $q \mapsto \overline{q}$ of the $\F$-algebra $\calQ$
(i.e., an involutory antiautomorphism) such that
$q\overline{q}=\overline{q}q$ and $q \overline{q} \in \F$ for all $q \in \calQ$, and we call it the quaternionic conjugation.
The mapping
$$N : q \in \calQ \mapsto q\overline{q} \in \F$$
is then the \textbf{quaternionic norm}, and
the mapping
$$t : q \in \calQ \mapsto q+\overline{q}$$
is valued in the center $\F$ and called the \textbf{(reduced) trace}.
In particular, we have the quadratic identity
\begin{equation}\label{eq:quadidentity}
\forall q \in \calQ, \quad q^2=\tr(q)\,q-N(q).
\end{equation}
It can be shown that $t$ is a central function on $\calQ$ (i.e.\ $t(xy)=t(yx)$ for all $x,y$ in $\calQ$) and that
it spans the $\F$-linear subspace of all central $\F$-linear forms on $\calQ$.
Beware that the Galois trace of $\calQ$ over $\F$, defined for $x \in \calQ$ as the trace of the $\F$-endomorphism $q \mapsto xq$ of $\calQ$,
equals twice the reduced trace $t$, and in particular vanishes if $\F$ has characteristic $2$ (a special case we do not want to discard here).
In particular $t(1_\calQ)=2.1_\F$.
The quaternionic norm is a regular quadratic form on the $\F$-vector space $\calQ$, with polar form $(a,b) \mapsto \tr(a\overline{b})$.

The kernel of the trace is the linear hyperplane $\calQ^p$ of \textbf{pure} quaternions, and it is known that it coincides
with the subgroup of $(\calQ,+)$ spanned by the Lie commutators $[a,b]$ with $a,b$ in $\calQ$. Better still (see Remark \ref{rem:commutator})
every pure quaternion \emph{is} a Lie commutator. If $\car(\F) \neq 2$ then $\F \oplus \calQ^p=\calQ$, but otherwise
$\F \subset \calQ^p$, which requires much caution. In any case, the trace induces an isomorphism of $\F$-vector spaces from $\calQ/\calQ^p$ to $\F$,
and hence for a matrix $A$ in $\Mat_n(\calQ)$ it is more convenient to define its reduced trace as $t(\Tr(A))$, which we do from now on.

If $\calQ$ is not a division ring then it is isomorphic to $\Mat_2(\F)$ (the ``split case", see e.g., \cite{Voight}),
and then $\Mat_n(\calQ)$ is isomorphic as an $\F$-algebra to $\Mat_{2n}(\F)$: hence this case is already covered by the field case.
Therefore, throughout the rest of the article we will assume that $\calQ$ is a division ring.
In that case, we can restate the recent result of de la Cruz et al \cite{dLC}:

\begin{theo}
Let $\calQ$ be a quaternion division algebra, and let $n \geq 2$.
Then every element of $\Mat_n(\calQ)$ with reduced trace zero is the sum of three nilpotent matrices,
and there exist elements of $\Mat_n(\calQ)$ with reduced trace zero that are neither scalar matrices nor the sum of two nilpotent matrices.
\end{theo}

Of course, for $n=1$ only the zero matrix is a sum of nilpotent ones, so the condition $t(M)=0$ is insufficient in that case.
Moreover, when $\calQ$ is Hamilton's real quaternion algebra $\mathbb{H}$, de la Cruz et al characterized the $2$-by-$2$ matrices that are sums of two nilpotent
matrices, and also arbitrary matrices that are sums of two square-zero matrices (i.e., nilpotent matrices with nilindex at most $2$).

The purpose of this article is to give a clear characterization of the elements of $\Mat_n(\calQ)$ that split into the sum of two nilpotent matrices.
At this point we cannot state this characterization, as it requires some elements of spectral theory (i.e.,
the notions of eigenvectors, eigenspaces and diagonalisability) over division rings. We review the necessary points in the next section, and will state our results afterwards.

\subsection{A quick review of spectral theory over division rings}\label{section:reviewspectral}

Throughout, we will use the endomorphisms viewpoint.
Let $D$ be a division ring, and denote its center by $C$. We write $a \simeq b$ to state that the elements $a$ and $b$
of $D$ are conjugated (i.e.\, $a=\gamma b\gamma^{-1}$ for some $\gamma \in D^\times$).
Throughout, we will consider finite-dimensional \emph{right} vector spaces over $D$
(rather than left vector spaces), and will simply call them finite-dimensional vector spaces over $D$.
For such vector spaces $V$ and $V'$, with respective bases $(e_1,\dots,e_p)$ and $(f_1,\dots,f_n)$, and
a linear mapping $u : V \rightarrow V'$, we define the matrix $M=(m_{i,j})$ of $u$ in these bases as the one
such that $u(e_j)=\sum_{i=1}^n f_i\,m_{i,j}$ for all $j \in \lcro 1,p\rcro$.
This way, we obtain the traditional way of interpreting $u$ in terms of matrix multiplication:
for a vector $x \in V$, with (column) coordinate vector $X$ in $(e_1,\dots,e_p)$, the coordinate vector of $u(x)$ in $(f_1,\dots,f_n)$ is $MX$.
For a matrix representing an endomorphism, we always take $(e_1,\dots,e_p)=(f_1,\dots,f_n)$.

Now, let $u$ be an endomorphism of a finite-dimensional $D$-vector space $V$.
Given two matrices $A$ and $B$ that represent $u$, the matrices $A$ and $B$ are similar, so
$\tr(A)=\tr(B)$ (remembering that $\Tr(A)-\Tr(B)$ is a sum of Lie commutators in $D$).
Hence $\tr(u):=\tr(A)$ is independent of the choice of a representing matrix $A$ for $u$.
This defines a $C$-linear mapping $\tr : \End_D(V) \rightarrow D/\calF$, where $\calF$ stands for the additive subgroup generated by
the Lie commutators $[a,b]$ with $a,b$ in~$D$.

We now turn to eigenvalues. An \textbf{eigenvalue} of $u$ is a scalar $q \in D$ for which the following equivalent conditions hold:
\begin{enumerate}[(i)]
\item The mapping $x \in V \mapsto u(x)-xq$ is noninjective;
\item There exists a nonzero vector $x \in V$ such that $u(x)=xq$. Such a vector is then called an \textbf{eigenvector} of $u$ attached to $q$.
\end{enumerate}
Equivalently, an eigenvector of $u$ is a nonzero vector $x$ such that the line $xD$ is invariant under $u$.

In general, an endomorphism $u$ can have infinitely many eigenvalues, even in the finite-dimensional case. In fact, if $q$ is an eigenvalue and $\vec{x}$
is a corresponding eigenvector then, for all $g \in D \setminus \{0\}$, the nonzero vector $\vec{x}g$ is an eigenvector of $u$ with corresponding eigenvalue $g^{-1} q g$,
and in particular the set of all eigenvectors of $u$ attached to $q$ is not a $D$-linear subspace unless $q$ is central.
The set of all eigenvalues of $u$ is a union of conjugacy classes in the ring $D$, and what is relevant is actually
the set of those conjugacy classes. Given pairwise nonconjugate eigenvalues $q_1,\dots,q_n$ and corresponding eigenvectors $x_1,\dots,x_n$, one proves that $x_1,\dots,x_n$
are $D$-linearly independent. Hence the set of all conjugacy classes of eigenvalues is finite with cardinality at most $\dim_D V$.
We will say that $u$ is \textbf{unispectral} whenever it has exactly one conjugacy class of eigenvalues.

We say that $u$ is \textbf{diagonalisable} if the following equivalent conditions hold:
\begin{enumerate}[(i)]
\item There is a basis of $V$ in which all the vectors are eigenvectors of $u$.
\item At least one diagonal matrix is a representing matrix for $u$.
\end{enumerate}
When we have a basis $(e_1,\dots,e_n)$ of eigenvectors of $u$, with respective corresponding eigenvalues $q_1,\dots,q_n$,
it is easy to prove that:
\begin{enumerate}[(i)]
\item The spectrum of $u$ is the union of the conjugacy classes of $q_1,\dots,q_n$;
\item Every eigenvector of $u$ belongs to the $D$-linear span of $(e_i)_{i \in I}$ for some nonempty subset $I$ of $\lcro 1,n\rcro$
such that the eigenvalues $q_i$, with $i \in I$, are pairwise conjugate.
\end{enumerate}
This shows that the diagonalisable unispectral endomorphisms are exactly those that are represented by a diagonal matrix with all diagonal
entries equal. Beware that these are not the scalar matrices of $\Mat_n(D)$, which represent the multiplications $x \mapsto x \lambda$ with $\lambda \in C$.

We will resume reviewing key results of spectral theory later and are now ready to present the main examples which lead to our central result.

\subsection{Key counterexamples, and the main theorem}

We will single out three kinds of special endomorphisms that fail to split into the sum of two nilpotent ones,
yet can have trace zero in some cases.
The first two are relevant for any division ring, so we keep the greatest level of generality and state
them for an arbitrary division ring $D$ with center $C$, and a finite-dimensional $D$-vector space $V$.

To start with, we have the same obstruction as in Theorem \ref{theo:basic} regarding scalar endomorphisms.
An endomorphism of \textbf{special type I} is simply a nonzero scalar endomorphism, i.e., an endomorphism of the form
$\id_V \lambda$ for some \emph{nonzero} $\lambda \in C$, and it is straightforward to see that no such endomorphism
can split into the sum of two nilpotent ones, as in a nontrivial ring the sum of a central invertible element with a nilpotent element is always invertible
(and hence non-nilpotent).

Assume now that $\dim V\geq 2$.
An endomorphism of $V$ is of \textbf{special type II} when it is a rank $1$ perturbation of a scalar endomorphism, i.e.,
the sum of a scalar endomorphism with a rank $1$ endomorphism.
If $\dim V \geq 3$, it is clear that such a decomposition is unique (because the only way a difference of two rank $1$
endomorphisms can be scalar is if they are equal).
Still in that case, and if $u=\id_V \lambda+a$ is of special type II with $\lambda \in C$ and $\rk a=1$,
the range $\im a$ of $a$ is $1$-dimensional and invariant under $a$, so all its nonzero vectors
are eigenvectors of $a$. The conjugacy class of eigenvalues attached to those vectors is called the
\textbf{supertrace} of $a$, denoted by $\str(a)$, and then it is obvious that $(\dim V).\lambda+\str(a)$
is also a conjugacy class, and we set
$$\str(u):=(\dim V).\lambda+\str(a)$$
and call it the supertrace of $u$.
If we take a representative $q$ of the supertrace of $u$, then $\tr(u)$ is the coset of $q$.
If $\dim V=2$, then the decomposition is not unique, but we will prove anyway in Section \ref{section:specialtypeII} that the definition of the supertrace
is independent of the chosen decomposition.

Then, we have a first important result, to be proved later (see Section \ref{section:specialtypeII}).

\begin{prop}\label{prop:specialtypeII}
Let $D$ be a division ring, and $u$ be an endomorphism of a finite-dimensional $D$-vector space $V$
with special type II.
The following conditions are then equivalent:
\begin{enumerate}[(i)]
\item The supertrace of $u$ is zero.
\item $u$ is represented by a matrix with diagonal zero.
\item $u$ is the sum of two nilpotent endomorphisms of $V$.
\end{enumerate}
\end{prop}

For example, if we take a pure quaternion $q \in \calQ^p$ that is not in $\F$, then the diagonal matrix
$\Diag(q,0,\dots,0)$ represents an endomorphism of special type II, with supertrace the conjugacy class of $q$,
and reduced trace zero. By Proposition \ref{prop:specialtypeII}, this matrix is not the sum of two nilpotent ones although its reduced trace equals zero.

Now, we consider a quaternion division algebra $\calQ$ over a field $\F$.
The last special case we will encounter is quite strange, as it happens only in the $3$-dimensional case.

\begin{prop}\label{prop:specialtypeIII}
Let $q \in \calQ \setminus \{0\}$. Then the diagonal matrix $\Diag(q,q,q)$ is not the sum of two nilpotent matrices of $\Mat_3(\calQ)$.
\end{prop}

Again, if we take $q$ as a pure quaternion outside of $\F$ (or even an arbitrary noncentral quaternion if $\car(\F) = 3$),
then $\Diag(q,q,q)$ has reduced trace zero. It is obvious that $\Diag(q,q,q)$ is of neither special type I nor special type II
(indeed, in any one of those cases the matrix under consideration has an eigenvalue in $\F$).

More generally, we say that an endomorphism $u$ of a $\calQ$-vector space $V$ is of \textbf{special type III} whenever $\dim V=3$
and $u$ is unispectral, diagonalisable and nonzero.

We are finally ready to state our main theorem:

\begin{theo}[Classification of sums of two nilpotent endomorphisms]\label{theo:n>=3}
Let $\calQ$ be a quaternion division algebra over a field $\F$,
and $u$ be an endomorphism of a $\calQ$-vector space with finite dimension $n \geq 3$.
Assume that $u$ is of neither one of special types I, II and III.
Then the following conditions are equivalent:
\begin{enumerate}[(i)]
\item $\tr(u)=0$;
\item $u$ is represented by a matrix with diagonal zero;
\item $u$ is the sum of two nilpotent endomorphisms.
\end{enumerate}
\end{theo}

Combining this result with Propositions \ref{prop:specialtypeII} and \ref{prop:specialtypeIII} yields the full characterization of the sums of two nilpotent
matrices of $\Mat_n(\calQ)$ for $n \geq 3$.
The case where $n=2$ is slightly different, with a less enlightening characterization: we wait until Section \ref{section:n=2}
to give the precise result on this case.

One infers the following corollary for which, strangely enough, we found ourselves unable to find a direct proof of the implication (ii) $\Rightarrow$ (i):

\begin{cor}
Let $\calQ$ be a quaternion division algebra over a field.
Let $V$ be a $\calQ$-vector space with finite dimension, and let $u \in \End_\calQ(V)$.
Then the following conditions are equivalent:
\begin{enumerate}[(i)]
\item $u$ is represented by a matrix with diagonal zero;
\item $u$ is the sum of two nilpotent endomorphisms.
\end{enumerate}
\end{cor}

\subsection{Organization of the article}

The remainder of the article is organized as follows.
In Section \ref{section:extendedreview}, we dive deeper into the spectral theory over division rings.
The $2$-dimensional case is then completely solved in Section \ref{section:n=2}.
In Section \ref{section:specialtypeII}, we analyze the endomorphisms of special type II, and as a result
we determine those that are the sum of two nilpotent endomorphisms (thereby proving Proposition \ref{prop:specialtypeII}).
In Section \ref{section:specialtypeIII}, we prove that no endomorphism of special type III is the sum of two nilpotent endomorphisms.

The last two sections are devoted to the proof of Theorem \ref{theo:n>=3}, which is the most difficult and intricate part of the article.
The proof is by induction, and it is initialized at $n=3$.
The initialization is not easy, and we gather the basic results needed for this in Section \ref{section:n=3}.
For the inductive step, the main idea is an adaptation of the classical inductive proof of Fillmore's theorem \cite{Fillmore}, where the difficulty is to avoid encountering
endomorphisms of special type when appealing to the induction hypothesis.

\section{Additional results of spectral theory over division rings}\label{section:extendedreview}

We resume our review of spectral theory over division rings.
Let $D$ be a division ring with center $C$, and $u$ be an endomorphism of a finite-dimensional $D$-vector space $V$.

It is important to recognize the eigenvalues of $u$ from a triangularization basis (when it exists).
So, let $(e_1,\dots,e_n)$ be a basis of $V$ in which $u$ is represented by some upper-triangular matrix $T$.
Then one proves that the spectrum of $u$ is included in the union of the conjugacy classes of the diagonal
entries $t_1,\dots,t_n$ of $T$: to see this, we take an arbitrary eigenvector $x$ of $u$, with eigenvalue denoted by $q$, we decompose it as $x=\sum_{i=1}^n e_i x_i$;
take the greatest $k \in \lcro 1,n\rcro$ such that $x_k \neq 0$; finally we infer that $t_k x_k=x_k q$, to the effect that $q \simeq t_k$.
Surprisingly, some diagonal entries of $T$ might not be eigenvalues of $u$, but fortunately in the case
of centrally finite division rings this important result holds.
Recall first that the division ring $D$ is \textbf{centrally finite} if it has finite degree over its center.

\begin{lemma}\label{lemma:triangulareigenvalues}
Let $D$ be a centrally finite division ring.
Let $M \in \Mat_n(D)$ be a matrix of the form $M=\begin{bmatrix}
S & X_0 \\
[0]_{1 \times (n-1)} & t
\end{bmatrix}$
with $S \in \Mat_{n-1}(D)$, $X_0 \in D^{n-1}$ and $t \in D$.

Then $t$ is an eigenvalue of $M$.
\end{lemma}

\begin{proof}
Assume first that $t$ is an eigenvalue of $S$. Then there is a nonzero vector $Y \in D^{n-1}$ such that $SY=Yt$.
Set $\widetilde{Y}:=\begin{bmatrix}
Y \\
0
\end{bmatrix}$ and note that $\widetilde{Y} \neq 0$ and $M \widetilde{Y}=\widetilde{Y}t$. Hence $t$ is an eigenvalue of $M$.

Assume now that $t$ is not an eigenvalue of $S$.
Then the endomorphism $X \in D^{n-1} \mapsto SX \in D^{n-1}$ does not have $t$ in its spectrum, i.e.,
$f : X \in D^{n-1} \mapsto SX-Xt$ is injective.

Denote by $C$ the center of $D$.
Here $D^{n-1}$ has finite dimension as a $C$-vector space, and hence
$f$, which is $C$-linear, is also surjective, yielding a vector $X' \in D^{n-1}$ such that $SX'-X' t=-X_0$.
Then we note that $\widetilde{X}:=\begin{bmatrix}
X' \\
1
\end{bmatrix}$ satisfies $M \widetilde{X}=\widetilde{X} t$. Hence $t$ is an eigenvalue of $M$.
\end{proof}

By induction on the size of the matrix, we deduce:

\begin{cor}\label{cor:triangulareigenvalues}
Let $D$ be a centrally finite division ring.
Let $T \in \Mat_n(D)$ be upper-triangular. Then every diagonal entry of $T$ is an eigenvalue of $T$.
\end{cor}

The reader must beware that many habits of the commutative case fall apart in the noncommutative one.
For example, a matrix of the form $\begin{bmatrix}
a & 1 \\
0 & a
\end{bmatrix}$ is never diagonalisable in the commutative case, but it might be diagonalisable in the noncommutative one!
Indeed, we have the following result:

\begin{lemma}\label{lemma:diago2by2}
Let $a$ and $b$ belong to a division ring $D$.
The matrix $M:=\begin{bmatrix}
a & b \\
0 & a
\end{bmatrix}$ of $\Mat_2(D)$ is diagonalisable if and only if $b=[a,x]$ for some $x \in D$.
\end{lemma}

\begin{proof}
If $b=[a,c]$ for some $c \in D$ then we observe that $TMT^{-1}=\Diag(a,a)$
for $T:=\begin{bmatrix}
1 & c \\
0 & 1
\end{bmatrix}$.

Conversely, assume that $M$ is diagonalisable. Then,
it has an eigenvector $X=\begin{bmatrix}
x_1 \\
x_2
\end{bmatrix}$ such that $x_2 \neq 0$. By considering $X x_2^{-1}$ we may directly assume that $x_2=1$.
Then the eigenvalue that corresponds to $X$ is $a$ (judging from the second entry in $MX$), and then we get $ax_1+b=x_1 a$ from the first entry,
to the effect that $b=[a,-x_1]$.
\end{proof}

In particular, whenever we have $a \in D \setminus C$, we can take $x \in D$ such that $[a,x] \neq 0$, and observe that the matrix
$\begin{bmatrix}
a & [a,x] \\
0 & a
\end{bmatrix}$ is diagonalisable although $[a,x] \neq 0$.

We finish with an important point on diagonalisable endomorphisms and the properties of invariant subspaces.
We give the proof because many standard proofs in the commutative case use tools that are not available
in the noncommutative one (such as the annihilating polynomials and B\'ezout's Theorem).

\begin{prop}\label{prop:induceddiagonalisable}
Let $u$ be a diagonalisable endomorphism of a finite-dimensional vector space $V$, and $W$ be an invariant linear subspace of $V$.
Then the respective endomorphisms of $V/W$ and $W$ induced by $u$ are diagonalisable (with spectrum included in the one of $u$).
\end{prop}

\begin{proof}
The easier statement is the one on the endomorphism $\overline{u}^W$ of $V/W$ induced by $u$.
Indeed, for every eigenvector $x$ of $u$, with corresponding eigenvalue $q$, one sees that the coset $\overline{x}$ in $V/W$
is either zero or an eigenvector of $\overline{u}^W$ with corresponding eigenvalue $q$. Taking a basis of such vectors, we
obtain a spanning set of eigenvectors of $\overline{u}$ for the vector space $V/W$ (we remove the zero vector, of course)
with all eigenvalues in the spectrum of $u$, and from this family we can extract a basis of $V/W$. Hence $\overline{u}^W$
is diagonalisable with spectrum included in the one of $u$.

For $W$, we use the trick of finding an invariant direct summand $W'$, that is a linear subspace such that $V=W \oplus W'$ and $W'$ is invariant under $u$:
this is obtained by choosing an arbitrary basis of $W$ and extending it into a basis of $V$ in which the remaining vectors $e_1,\dots,e_d$ are chosen
among a basis of eigenvectors of $u$. Then $W':=e_1D+\cdots+e_d D$ clearly fulfills our needs. Then, we observe that the endomorphism $u_W$ of $W$
induced by $u$ is similar to the endomorphism of $V/W'$ induced by $u$. By the first step, the latter is diagonalisable, and hence so is $u_W$,
and of course all the eigenvalues of $u_W$ are eigenvalues of $u$.
\end{proof}

\section{The $2$-dimensional case}\label{section:n=2}

Here we deal with the characterization of sums of two nilpotent endomorphisms of $2$-dimensional spaces.
As we are about to see, our ``spectral" characterization (condition (iii) below) is quite complicated.
Before we prove it, we need a classical lemma which will be reused later:

\begin{lemma}\label{lemma:carachomotheties}
Let $V$ be a vector space with dimension at least $2$ over a division ring $D$, and let $u$ be an endomorphism of $V$.
Assume that every vector of $V$ is an eigenvector of $u$. Then $u=\id_V \lambda$ for some $\lambda$ in the center of $D$.
\end{lemma}

\begin{proof}
For every $x \in V \setminus \{0\}$, denote by $q_x \in D$ the unique scalar such that $u(x)=x q_x$.
Let $x,y$ be linearly independent vectors of $V$. Then by applying the property to $x$, $y$ and $x+y$ we find that
$q_x=q_{x+y}=q_y$ because $u$ is additive.

Let now $x,y$ be nonzero vectors of $V$ that are linearly dependent over $D$. Then we can find $z \in V \setminus xD$,
and we deduce from the first step that $q_y=q_z=q_x$. Hence $x \mapsto q_x$ is constant, and we denote its value by $\lambda$, which at this point belongs to $D$.
Finally, for all $g \in D^\times$ we compute that $\forall x \in V \setminus \{0\}, \; q_{x g}=g^{-1} q_x g$, and hence $\lambda$ belongs to the center of $D$.
\end{proof}

\begin{theo}\label{theo:n=2}
Let $\calQ$ be a quaternion division algebra over a field $\F$.
Let $V$ be a $2$-dimensional $\calQ$-vector space, and $u \in \End(V)$.
The following conditions are equivalent:
\begin{enumerate}[(i)]
\item $u$ is the sum of two nilpotent endomorphisms;
\item $u$ is represented by a matrix with diagonal zero;
\item $u^2$ is unispectral and diagonalisable, and if $u$ is unispectral and diagonalisable then either it equals zero or its eigenvalues are noncentral pure quaternions.
\end{enumerate}
\end{theo}

\begin{proof}
As usual, the implication (ii) $\Rightarrow$ (i) is obvious.

\vskip 2mm
\noindent \emph{Proof of} (i) $\Rightarrow$ (iii). \\
Assume that $u=n_1+n_2$ for nilpotent endomorphisms $n_1$ and $n_2$ of $V$.
If $n_1$ and $n_2$ have a common eigenvector $x$, then we extend it to a basis $(x,y)$ and we find that $u$ is represented by a matrix of the form $\begin{bmatrix}
0 & ? \\
0 & 0
\end{bmatrix}$ in this basis. Then $u^2=0$ so $u^2$ is unispectral and diagonalisable, and if $u$ is unispectral it is zero (because obviously $0$ is its sole eigenvalue).

Assume next that $n_1$ and $n_2$ have no common eigenvector. Then, pick $x \in \Ker n_2 \setminus \{0\}$ and $y \in \Ker n_1 \setminus \{0\}$,
so that $(x,y)$ is a basis of $V$. The respective matrices of $n_1$ and $n_2$ in this basis are
$N_1=\begin{bmatrix}
0 & 0 \\
a & 0
\end{bmatrix}$ and $N_2=\begin{bmatrix}
0 & b \\
0 & 0
\end{bmatrix}$, with $a \neq 0$ and $b \neq 0$ (otherwise $n_1$ and $n_2$ would have a common eigenvector). Then
the matrix of $u^2$ in that basis is $\Diag(ba,ab)$, and $ab=a(ba)a^{-1}$ so $u^2$ is unispectral and diagonalisable, as claimed.
Assume, on top of the previous assumption on $n_1$ and $n_2$, that $u$ is unispectral and diagonalisable, and denote by $q$ a corresponding eigenvalue.
Then $q^2=t(q)q-N(q)$ by the quadratic identity (see \eqref{eq:quadidentity}), so $u^2=u\lambda+\id_V \mu$ for $\lambda:=t(q)$ and $\mu:=-N(q)$.
Then $x(ba)=u^2(x)=u(x)\lambda+x\mu$, so $\lambda=0$. Hence $q$ is a pure quaternion.
If $q \in \F$ then $u$ is scalar, and we must have $u=0$ because $u$ is the sum of two nilpotent endomorphisms.
Hence $q \not\in \F$ or $u=0$. This proves that condition (iii) holds.

\vskip 2mm
\noindent \emph{Proof of} (iii) $\Rightarrow$ (ii). \\
Assume that condition (iii) holds.
Choose an eigenvector $x$ of $u^2$, with eigenvalue denoted by $q$.
If $x,u(x)$ are linearly independent, then the matrix of $u$ in the basis $(x,u(x))$ is $\begin{bmatrix}
0 & q \\
1 & 0
\end{bmatrix}$, and its diagonal is zero.

Assume now that every eigenvector of $u^2$ is an eigenvector of $u$.
Consider a basis $(x_1,x_2)$ of eigenvectors of $u^2$ with the same corresponding eigenvalue $q$.
Then $(x_1,x_2)$ is a basis of eigenvectors of $u$, and we write the corresponding matrix $\Diag(q_1,q_2)$.
Note that $x_1+x_2$ is an eigenvector of $u^2$, so it is also an eigenvector of $u$. As seen in the last part of Section \ref{section:reviewspectral},
this requires that $q_1 \simeq q_2$. Hence $u$ is unispectral and diagonalisable, and our assumptions require
then that either $q_1=0$ or $q_1$ is a noncentral pure quaternion.
In any case $q_1^2 \in \F$, and since $u^2$ is unispectral and diagonalisable it is a scalar endomorphism.
Hence every nonzero vector of $V$ is an eigenvector of $u^2$, and hence of $u$ by our assumptions.
By Lemma \ref{lemma:carachomotheties}, it follows that $u$ is scalar, and hence $q_1 \in \F$. Thus $q_1=0$
and we conclude that $u=0$, in which case (ii) is trivially satisfied.
\end{proof}

\section{Endomorphisms of special type II}\label{section:specialtypeII}

The aim of this section is to prove Proposition \ref{prop:specialtypeII}, i.e.,
determine which endomorphisms of special type II are sums of two nilpotent endomorphisms.
Throughout, $D$ denotes an arbitrary division ring with center denoted by $C$.

To start with, we must prove that the supertrace of an endomorphism of special type II is correctly defined.

\begin{lemma}
Let $V$ be a finite-dimensional $D$-vector space.
Let $\lambda,\mu$ belong to $C$, and let $a$ and $b$ be rank $1$ endomorphisms of $V$
such that $\id_V \lambda+a=\id_V\mu+b$. Then $(\dim V).\lambda+\str(a)=(\dim V).\mu+\str(b)$.
\end{lemma}

\begin{proof}
The statement is obvious if $\dim V=1$, since in that case we take an arbitrary $x \in V \setminus \{0\}$
to compute the supertrace of $a$ as the only $q \in D$ such that $a(x)=xq$, and ditto with $b$ and the same $x$.

Now, let us assume that $\dim V \geq 2$. The statement is trivial if $a=b$.
So, let us assume that $a \neq b$, to the effect that $\id_V (\lambda-\mu)=b-a$ with $\lambda-\mu \neq 0$.
Then $\dim V = \rk(b-a) \leq 2$, and hence $\dim V=2$.
Classically, it also follows that $V=\im a \oplus \im b$ because $a$ and $b$ have rank $1$.
Finally, choose $x \in \im a \setminus \{0\}$. Then $x(\lambda-\mu)=b(x)-a(x)$ and we deduce that $b(x)=0$ and
$\str(a)=\{\mu-\lambda\}$. Symmetrically $\str(b)=\{\lambda-\mu\}$, and since
$2\lambda+\{\mu-\lambda\}=\{\lambda+\mu\}=2\mu+\{\lambda-\mu\}$ the claimed statement is proved.
\end{proof}

Now that the supertrace has been correctly defined, we move on to the proof of Proposition \ref{prop:specialtypeII}.
In this proposition, the implication (ii) $\Rightarrow$ (iii) is, as always, obvious.
To start with, we will prove that every endomorphism of special type II with supertrace $0$
is represented by a matrix with diagonal zero.

So, let $V$ be a vector space over $D$ with finite dimension $n \geq 1$.
Let $a \in \End(V)$ be of rank $1$, and let $\lambda \in C$. We set $u:=\id_V \lambda+a$ and assume that
$\str(u)=\{0\}$, to the effect that $\str(a)=\{\mu\}$ for $\mu:=-n.\lambda$.
If $n=1$ it is then straightforward that $u=0$ and hence $u$ is represented by the zero matrix.
We assume from now on that $n \geq 2$.

The simple idea is to prove that, in some basis, $u$ is represented by a nonscalar matrix of $\Mat_n(C)$ with trace $0$.
Then, since $C$ is a field Theorem \ref{theo:basic} will yield that $u$ is represented by a matrix of $\Mat_n(C)$ with diagonal zero.

To start with, we claim that in some basis of $V$ the matrix of $a$ equals
$$\begin{bmatrix}
\mu & [?]_{1 \times n-1} \\
[0]_{(n-1) \times 1} & [0]_{n-1}
\end{bmatrix}$$
\emph{and all the entries of this matrix belong to $C$.}
\begin{itemize}
\item If $\mu \neq 0$, it suffices to note that $V=\im(a) \oplus \Ker(a)$ and to take a basis that is adapted to this decomposition, in which case in the first row all the entries are zero starting from the second one.
\item If $\mu=0$, we choose $e_1 \in \im a \setminus \{0\}$, then $e_2 \in V$ such that $a(e_2)=e_1$, we note that $e_1$ and $e_2$ are linearly independent over $D$,
and we extend $(e_1,e_2)$ into a basis of $V$ by using vectors from $\Ker a$ (which works because $\Ker a \oplus e_2 D=V$).
\end{itemize}
Therefore, in the same basis the endomorphism $u$ is represented by a matrix $M$ of $\Mat_n(C)$
with trace $\mu+n\lambda$, i.e., with trace zero. We must check that $M$ is nonscalar but this is easy:
indeed, if $M$ is scalar then $u=\id_V \alpha$ for some $\alpha \in C$ and we recover $\id_V(\mu-\alpha)=a$, contradicting the fact that $\dim V>1$ and $\rk(a)=1$.
Hence, $M$ is nonscalar with trace $0$. Applying Theorem \ref{theo:basic} with the field $C$, we conclude that $u$ is represented by a matrix with diagonal zero.

\vskip 3mm
Next, we deal with the implication (iii) $\Rightarrow$ (i) in Proposition \ref{prop:specialtypeII}.
This is an adaptation of the proof of (5.1) in \cite{dLC}, and we include it for the sake of completeness.
The proof works by induction on $n$. The case $n=1$ is obvious: if an endomorphism of special type II of a $1$-dimensional vector space is zero, then it is clear that its supertrace is zero.

Now, we let $V$ be an $n$-dimensional vector space over $D$, and $u \in \End_D(V)$ be of special type II, and we write
$u=\id_V \lambda+a$ for some $\lambda \in C$ and some $a \in \End_D(V)$ of rank $1$. And we assume that there are nilpotent endomorphisms
$n_1$ and $n_2$ of $V$ such that $u=n_1+n_2$. Our aim is to prove that $\str(u)=\{0\}$.

We start by choosing an arbitrary nonzero vector $x$ in $\Ker n_2$.
Assume first that $n_1(x)=0$. Then $u(x)=0$, and so $a(x)=-x\lambda$.
Then $n_1$ and $n_2$ induce respective nilpotent endomorphisms $\overline{n_1}$ and $\overline{n_2}$ of $V/x D$;
$u$ and $a$ induce respective endomorphisms $\overline{u}$ and $\overline{a}$ of $V/x D$;
we have $\overline{u}=\id \lambda+\overline{a}$ and $\overline{u}=\overline{n_1}+\overline{n_2}$,
and finally $\rk(\overline{a}) \leq 1$.
\begin{itemize}
\item Assume that $\overline{a}=0$, i.e.\ $\im a=x D$. Then $\overline{u}$ is scalar so we must have $\lambda=0$.
Then $a(x)=0$ and hence $\str(a)=\{0\}$, to the effect that $\str(u)=\{0\}$.

\item Assume that $\overline{a} \neq 0$. Then $\rk \overline{a}=1$, and by induction we gather that
$\str(\overline{u})=\{0\}$, so $\str(\overline{a})=\{-(n-1)\lambda\}$.
Besides $a(x)=- x \lambda$ leads to $\lambda=0$
(otherwise $\im a=x D$ and we would find $\overline{a}=0$), whence $\str(\overline{a})=\{0\}$
and $\str(u)=\str(a)$. We will conclude by noting that $\str(a)=\str(\overline{a})$.
To see this, we take $y \in (\im a) \setminus \{0\}$ and write $a(y)=yq$ for some $q \in D$, to the effect that $\str(a)$
is the conjugacy class of $q$; Then, if $\str(a) \neq \{0\}$ we deduce that $q \neq 0$, so the coset of $y$ mod $xD$ spans the range of $\overline{a}$,
to the effect that $\str(\overline{a})$ is the conjugacy class of $q$, yielding the claimed equality;
Otherwise $\str(a)=\{0\}=\str(\overline{a})$. Hence $\str(a)=\str(\overline{a})$ in any case, and we conclude that $\str(u)=\{0\}$.
\end{itemize}

From now on, we assume that $n_1(x) \neq 0$ and we consider the local nilindex $d$ of $x$ with respect to $n_1$,
that is the least positive integer such that $n_1^d(x)=0$. Classically $x,n_1(x),\dots,n_1^{d-1}(x)$ are linearly independent
and the $D$-linear subspace $W:=\Vect\bigl(x,n_1(x),\dots,n_1^{d-1}(x)\bigr)$ is invariant under $n_1$.
Note in particular that $a(x)=-x \lambda+n_1(x)$ belongs to $W$ and is nonzero, so $\im(a) \subseteq W$ because $\rk(a)=1$.
It follows that $W$ is invariant under $a$, and we conclude that it is also invariant under $n_2=u-n_1=\id_V\lambda+a-n_1$, and of course also under $u$.
Now, we have to split the discussion into two subcases.

\vskip 2mm
\noindent \textbf{Case 1: $d<n$.} \\
All the endomorphisms $u,a,n_1,n_2$ induce endomorphisms of $W$, denoted respectively by $u_W,a_W, (n_1)_W$ and $(n_2)_W$.
The latter two are nilpotent, and we have $u_W=(n_1)_W+(n_2)_W$, so by induction $\str(u_W)=\{0\}$.
We have seen that $a(x) \neq 0$ so $a_W \neq 0$, to the effect that $\rk(a_W)=1$.
Then $\im(a_W) \subseteq \im(a)$ with $\rk(a)=1=\rk(a_W)$, leading to $\im(a_W)=\im(a)$.
It is then obvious that $\str(a_W)=\str(a)$.

Likewise, we obtain induced endomorphisms $u^{V/W},a^{V/W}, (n_1)^{V/W}$ and $(n_2)^{V/W}$ of the quotient space $V/W$, with $a^{V/W}=0$,
and $(n_1)^{V/W}$ and $(n_2)^{V/W}$ nilpotent, and $\id_{V/W} \lambda=u^{V/W}=(n_1)^{V/W}+(n_2)^{V/W}$. Hence $\lambda=0$ because $\dim(V/W)>0$.
It follows that
$$\str(u)=\str(a)=\str(a_W)=\str(u_W)=\{0\}.$$

\vskip 2mm
\noindent \textbf{Case 2: $d=n$.} \\
Then $(x,n_1(x),\dots,(n_1)^{d-1}(x))$ is a basis of $V$.
The respective matrices of $n_1$ and $n_2$ in this basis have the following forms
$$N_1=\begin{bmatrix}
[0] & [0]_{1 \times (n-1)} \\
[?]_{(n-1) \times 1} & N'_1
\end{bmatrix} \quad \text{and} \quad N_2=\begin{bmatrix}
[0] & [?]_{1 \times (n-1)} \\
[0]_{(n-1) \times 1} & N'_2
\end{bmatrix}$$
with $N'_1$ and $N'_2$ in $\Mat_{n-1}(D)$, and obviously both $N'_1$ and $N'_2$ are nilpotent.
The matrix of $a$ in the same basis looks as follows:
$$A=\begin{bmatrix}
-\lambda & [?]_{1 \times (n-1)} \\
[?]_{(n-1) \times 1} & A'
\end{bmatrix},$$
where $\rk A' \leq 1$ and all the rows in $A'$ are zero starting from the second one (recall that $\im(a)=a(x) D \subset \Vect(x,n_1(x))$).
Hence $I_{n-1} \lambda +A'$ is the sum of two nilpotent matrices, and we can apply the inductive hypothesis
to it if $A'$ has rank $1$ (which might not be the case!).

Remember that $a(x)$ is a nonzero vector of $\im(a)$, and consider the associated eigenvalue $q$ for $a$.
The equality $a^2(x)=a(x)q$ leads to $a(-x\lambda+n_1(x))=a(x)q$ and hence to
$$a(n_1(x))=a(x)(q+\lambda)=-x\lambda(q+\lambda)+n_1(x)(q+\lambda).$$
Hence, the upper-left entry of $A'$ equals $q+\lambda$.
Now, we complete the proof by splitting the discussion into two cases:
\begin{itemize}
\item Assume that $A'=0$. Then on the one hand $\lambda I_{n-1}$ is the sum of two nilpotent matrices, which leads to $\lambda=0$;
on the other hand $q+\lambda=0$, so $q=0$ and $\str(u)=\str(a)=\{0\}$.
\item Assume finally that $A' \neq 0$. Then the column space of $A'$ is spanned by the first standard basis vector,
and it is then clear that the conjugacy class of $q+\lambda$ is the supertrace of $A'$. It follows by induction that $(n-1)\lambda+q+\lambda=0$, that is $q=-n\lambda$, and hence $\str(u)=\{0\}$.
\end{itemize}

The proof is completed.

\section{Endomorphisms of special type III}\label{section:specialtypeIII}

Here, we let $\calQ$ be a quaternion division algebra over a field $\F$.
We will prove the following result:

\begin{prop}
Let $V$ be a vector space over $\calQ$, and $u \in \End(V)$
be an endomorphism of special type III. Then $u$ is not the sum of two nilpotent endomorphisms of $V$.
\end{prop}

\begin{proof}
To start with, we can immediately discard the case where the spectrum of $u$ contains an element of $\F$,
as then $u$ would be scalar and nonzero. So, from now on we assume that $u$ has no eigenvalue in $\F$.

The proof will be based on two key observations. The first one is that $u$ is an automorphism of $V$.
The second one is that there are scalars $\lambda,\mu$ in $\F$ such that $u^2=u\lambda+\id_V \mu$.
Indeed, choose an eigenvalue $q$ of $u$, so that $u$ is represented by the matrix $\Diag(q,q,q)$ in some basis.
Then $q^2=q t(q)-N(q)$ where $t(q)$ and $N(q)$ belong to $\F$, and hence it is clear that $\lambda:=t(q)$ and $\mu:=-N(q)$ qualify.

Now, we assume that $u=n_1+n_2$ for nilpotent endomorphisms $n_1$ and $n_2$ of $V$, and we seek a contradiction.

Our first observation is that $n_1$ and $n_2$ have no common nontrivial invariant subspace.
Indeed, assume that such a subspace $W$ exists: then either it has dimension $1$, in which case $n_1$ and $n_2$ must vanish on it,
so $W \subset \Ker u$; or it has dimension $2$, in which case $n_1$ and $n_2$ induce nilpotent endomorphisms of $V/W$,
which must then equal zero, and hence $u$ maps into $W$. In any case we contradict the fact that $u$ is an automorphism.

Now, we choose $x \in \Ker n_2 \setminus \{0\}$. Then $x \calQ$ is not invariant under $n_1$, so $x,n_1(x)$ are linearly independent.
Note that $u(x)=n_1(x)$. Since $u^2=u \lambda+\id_V \mu$, we deduce that $u(n_1(x))$ is a linear combination of $x$ and $n_1(x)$.
In particular $\Vect(x,u(x))$ is invariant under $u$, and hence cannot also be invariant under $n_1$ (otherwise it would be invariant under $u-n_1=n_2$).
It follows that $x,n_1(x),n_1^2(x)$ are linearly independent, and in particular the range of $n_1$ is $\Vect(n_1(x),n_1^2(x))$ because $n_1$ is singular.
The respective matrices of $n_1$ and $n_2$ in the basis $(x,n_1(x),n_1^2(x))$ then have the following forms:
$$N_1=\begin{bmatrix}
0 & [0]_{1 \times 2} \\
[?]_{2 \times 1} & N'_1
\end{bmatrix} \quad \text{and} \quad N_2=\begin{bmatrix}
0 & [?]_{1 \times 2} \\
[0]_{2 \times 1} & N'_2 \\
\end{bmatrix},$$
with both $N'_1$ and $N'_2$ nilpotent in $\Mat_2(\calQ)$, whereas the one of $u$ takes the form
$$M=\begin{bmatrix}
0 & \mu & ? \\
1 & \lambda & ? \\
0 & 0 & \gamma
\end{bmatrix} \quad \text{for some $\gamma \in \calQ$.}$$
We observe from Lemma \ref{lemma:triangulareigenvalues} that $\gamma$ is an eigenvalue of $u$, and hence $\gamma \not\in \F$
because of our starting assumptions.

By extracting the lower $2$-by-$2$ block, we conclude that some matrix of the form
$K:=\begin{bmatrix}
\lambda & ? \\
0 & \gamma
\end{bmatrix}$ is the sum of two nilpotent matrices. Yet $K$ represents an endomorphism of special type II and whose supertrace contains
$\lambda+\gamma$. To see this, we simply take the matrix $K-I_2 \lambda$ and observe that is has rank $1$
and obviously $\gamma-\lambda$ in its spectrum (see Lemma \ref{lemma:triangulareigenvalues}), and $\gamma-\lambda \neq 0$.
It then follows from Proposition \ref{prop:specialtypeII} that $\lambda+\gamma=0$, which contradicts the fact that $\gamma \not\in \F$.
This final contradiction completes the proof.
\end{proof}

\section{Tools for the last proof}

Here we prove several lemmas that are required in the final proof. Throughout, we let $\calQ$ be a quaternion division algebra over a field $\F$.

\subsection{A key result for the case $n=2$}

The dimension $3$ case will require a result that is implicit in \cite{dLC} but which we must prove:

\begin{lemma}\label{lemma:completionn=2}
Let $(a,b) \in \calQ^2$ be such that $t(a+b)=0$. Then there exists $\delta \in \calQ$ such that the matrix
$$\begin{bmatrix}
a & \delta \\
1 & b
\end{bmatrix}$$
is the sum of two square-zero matrices of $\Mat_2(\calQ)$.
\end{lemma}

The proof requires the following classical lemma on quaternion algebras, for which we refer to \cite{dLC}:

\begin{lemma}\label{lemma:conjugate}
Let $q,q'$ be elements of $\calQ \setminus \F$. Then $q \simeq q'$ if and only if $t(q)=t(q')$ and $N(q)=N(q')$.
\end{lemma}

We will use the following application of this lemma:

\begin{lemma}\label{lemma:translateconjugate}
Let $a$ and $b$ belong to $\calQ$ and be such that $t(a)=t(b)$. Then there exists $q \in \calQ$ such that $a+q \simeq b+q$.
\end{lemma}

\begin{proof}
If $a=b$ the result is trivial. Assume now that $a \neq b$.

Let $q \in \calQ$.
We note that $t(a+q)=t(a)+t(q)=t(b)+t(q)=t(b+q)$. Denote by $\langle -,-\rangle$ the polar form of the norm (i.e., the quaternionic inner product).
Then $N(b+q)-N(a+q)=N(b)-N(a)+\langle b-a,q\rangle$.
Since $b-a \neq 0$ and $\langle -,-\rangle$ is nondegenerate, the solution set of the equation
$N(b)-N(a)+\langle b-a,q\rangle=0$ is an affine hyperplane $\calH$ of the $\F$-vector space $\calQ$.
The $1$-dimensional $\F$-affine subspaces $-a+\F$ and $-b+\F$ cannot cover $\calH$ (since it has dimension $3$), and we deduce that we can choose
$q \in \calH$ such that $a+q \not\in \F$ and $b+q \not\in \F$. Then $a+q \simeq b+q$ by Lemma \ref{lemma:conjugate}.
\end{proof}

\begin{Rem}\label{rem:commutator}
Take a pure quaternion $a \in \calQ^p$. Applying Lemma \ref{lemma:translateconjugate} to $b=0$ yields $q \in \calQ$ and
$u \in \calQ^\times$ such that $a+q=u q u^{-1}$, leading to $a=[uq,u^{-1}]$. Hence every pure quaternion is a Lie commutator.
\end{Rem}

Now we can prove Lemma \ref{lemma:completionn=2}.

\begin{proof}[Proof of Lemma \ref{lemma:completionn=2}]
Set $A:=\begin{bmatrix}
a & 0 \\
1 & b
\end{bmatrix}$ and $B(\delta):=\begin{bmatrix}
0 & \delta \\
0 & 0
\end{bmatrix}$.
By Lemma \ref{lemma:translateconjugate}, there exist $q \in \calQ$ and $g \in \calQ^\times$ such that $a+q=g (-b+q) g^{-1}$.
Set $D:=\Diag(1,g)$ and
$P:=\begin{bmatrix}
1 & q \\
0 & 1
\end{bmatrix}$, and note that
$$(D P) A (DP)^{-1}=\begin{bmatrix}
a+q & c \\
g & g(b-q)g^{-1}
\end{bmatrix}$$
for some $c \in \calQ$.
Note also that $(DP) B(\delta) (DP)^{-1}=B(\delta g^{-1})$.
Then, we set $s:=a+q$ and take $\delta \in \calQ$ such that $\delta g^{-1}+c=s^2$, so that
$$(D P) (A+B(\delta)) (DP)^{-1}=\begin{bmatrix}
s & s^2 \\
g & -s
\end{bmatrix}=\begin{bmatrix}
s & s^2 \\
-1 & -s
\end{bmatrix}+\begin{bmatrix}
0 & 0 \\
g+1 & 0
\end{bmatrix}.$$
Hence $A+B(\delta)$ is the sum of two square-zero matrices, as claimed.
\end{proof}

\subsection{A classical lemma on rank $1$ matrices}

\begin{lemma}\label{lemma:sumoftworank1}
Let $D$ be a division ring, and $a$ and $b$ be rank $1$ linear mappings from $U$ to $U'$, where $U$ and $U'$ are $D$-vector spaces.
If $\im(a) \neq \im(b)$ and $\Ker(a) \neq \Ker(b)$ then $\rk(a+b)>1$.
\end{lemma}

\begin{proof}
Assume that $\im(a) \neq \im(b)$ and $\Ker(a) \neq \Ker(b)$.
Since $\Ker(a)$ and $\Ker(b)$ are distinct $D$-linear hyperplanes of $U$, we can find $x \in \Ker(a) \setminus \Ker(b)$ and $y \in \Ker(b) \setminus \Ker(a)$.
Then $(a+b)(x) \in \im(a) \setminus \{0\}$ and $(a+b)(y) \in \im(b) \setminus \{0\}$, so $(a+b)(x)$ and $(a+b)(y)$ are linearly independent vectors of
$\im(a+b)$.
\end{proof}

\subsection{A local linear dependence lemma}

The next lemma is much more difficult. It explores the obstructions for the local linear independence of the family $(\id_V,u,u^2)$.

\begin{lemma}\label{lemma:3LLD}
Let $u$ be an endomorphism of a $3$-dimensional vector space $V$ over $\calQ$.
Assume that $x,u(x),u^2(x)$ are linearly dependent for all $x \in V$.
Then $u$ is of special type I, II or III.
\end{lemma}

Before we prove this lemma, we start with a variation of Lemma \ref{lemma:carachomotheties}:

\begin{lemma}\label{lemma:carachomothetiesmodhyperplane}
Let $D$ be a division ring with more than $2$ elements, $V$ be a $D$-vector space with dimension greater than $1$, and
$W$ be a proper linear subspace of $V$.
Let $u \in \End_D(V)$ be such that $u(x)$ and $x$ are linearly dependent for all $x \in V \setminus W$.
Then $u$ is of special type I.
\end{lemma}

\begin{proof}
Denote by $C$ the center of $D$.
Let $x,y$ in $V \setminus W$ be linearly independent, and consider the corresponding eigenvalues $q_x$ and $q_y$.
Assume first that $q_x$ and $q_y$ are not conjugates in $D$.
Then it is known (see Section \ref{section:reviewspectral}) that every eigenvector of $u$ in $\Vect(x,y)$ belongs to $x D \cup y D$.
It follows from our assumptions that $\Vect(x,y)=x D \cup y D \cup (W \cap \Vect(x,y))$,
thereby covering $\Vect(x,y)$ by two or three $1$-dimensional $D$-linear subspaces.
Since $|D|>2$, this is not possible.

Hence $q_x \simeq q_y$. By scaling $x$ and $y$, we can assume that $q_x=q_y=q$ for some $q \in D$.
Assume now that $q$ is not a central element of $D$.
Choose $q' \in D$ that does not commute with $q$. Then $x q'+y$ is not an eigenvector of $u$,
otherwise the corresponding eigenvalue would be $q$ (thanks to the coefficient on $y$ and the linear independence of $x$ and $y$)
and then $qq'=q'q$ by extracting the coefficient on $x$. Hence $x q'+y \in W$. Likewise $xq'q+y \in W$
(note that $q'q$ does not commute with $q$ because $q$ is invertible).
Hence $x(q'q-q') \in W$ by subtracting. Since $q'q-q'=q'(q-1)$ is invertible (because $q$ is not a central element),
we deduce that $x \in W$, which contradicts our starting assumptions. Therefore $q \in C$.

Now, we can conclude. Let $x \in V \setminus W$. Since $\dim_D V>1$ we can find $y \in V \setminus (W \cup x D)$,
and hence by the above the eigenvalue that is attached to $x$ is a central element $\lambda_x \in C$.
For all $y \in V \setminus W$, either $y$ is linearly dependent of $x$, then $\lambda_y=\lambda_x$
because $\lambda_x$ is central, or it is not and the previous step of the proof shows that $\lambda_x=\lambda_y$.

We conclude that there is a central element $\lambda \in C$ such that $u(x)=x\lambda$ for all $x \in V \setminus W$.
Hence, for $v : x \mapsto u(x)-x\lambda$, which is $C$-linear, we have $\Ker v \cup W=V$ where $\Ker v$ and $W$ are proper $\F$-linear subspaces of $V$,
which requires that $\Ker v=V$.
Hence $v=0$ and the conclusion follows.
\end{proof}

We are now in the position to prove Lemma \ref{lemma:3LLD}.

\begin{proof}[Proof of Lemma \ref{lemma:3LLD}]
We assume that $u$ is not of special type I and we will prove that it is of one of special types II or III.

First of all, since $u$ is not of special type I, we deduce from Lemma \ref{lemma:carachomotheties} that
there exists $x \in V \setminus \{0\}$ such that $y:=u(x)$ is not in $x\calQ$.
Then $P:=\Vect(x,y)$ is $2$-dimensional and is invariant under $u$, and obviously
$x$ is a nonzero vector in $P$ that is not an eigenvector of $u$.

Hence, we have found a $2$-dimensional $\calQ$-linear subspace $P$ of $V$ that is invariant under $u$ and in which
not all the nonzero vectors are eigenvectors of $u$. We fix such a space from now on, and forget about how it was constructed.

Next, we consider an arbitrary vector $z \in V \setminus P$. Then $u(z)-z\,q \in P$ for some $q \in \calQ$, and we set $x':=u(z)-z\,q$.
Because of our starting assumption $\Vect(z,u(z))$ is $u$-invariant. Since $P$ is also $u$-invariant, we find that $P \cap \Vect(z,u(z))$
is $u$-invariant, and we note that it has dimension at most $1$. It follows that either $x'$ equals $0$ or it is an eigenvector of $u$ in $P$.
To conclude $u(z)=z\,q+x'$ for some $q \in \calQ$ and some $x'\in P$ which is either $0$ or an eigenvector of $u$.
Of course $x' \neq 0$ if and only if $z$ is not an eigenvector of $u$.

Now, by Lemma \ref{lemma:carachomothetiesmodhyperplane} we can find a vector $z \in V \setminus P$ that is not an eigenvector of $u$,
and we fix it once and for all.

We write $u(z)=z\,q+x'$ for some $q \in \calQ$ and some $x'\in P$ which is an eigenvector of $u$.
Let $y \in P$ be arbitrary. Then $u(z+y)=(z+y)q'+y'$ for some $q' \in \calQ$ and some $y' \in P$ which is either zero or an eigenvector of $u$.
Yet $u(z+y) \equiv zq$ mod $P$, so $q'=q$ and we deduce that $u(y)=yq+y'-x'$.
Hence, for all $y \in P$ the vector $u(y)-yq+x'$ is either zero or an eigenvector of $u$ in $P$.
In the remainder of the proof, it will be crucial to observe that
$\{u(y)-yq+x' \mid y \in P\}$ is an $\F$-affine subspace of $P$ with dimension $8-\dim_\F \Ker(y \in P \mapsto u(y)-yq)$.
If $y \in P \mapsto u(y)-yq$ is injective, i.e., if $q$ is not an eigenvalue of $u_P$,
then we recover that $\{u(y)-yq+x' \mid y \in P\}=P$ and we deduce that every nonzero vector of $P$ is an eigenvector of $u$.
Yet this is known to be false. Therefore $q$ is an eigenvalue of $u_P$.

From there, we need to split the discussion between two subcases, whether
$q$ belongs to $\F$ or not.

\vskip 3mm
\noindent \textbf{Case 1.} $q \in \F$. \\
We write $\lambda:=q$ for greater clarity.
Here $u_P-\id_P \lambda$ is $\calQ$-linear and $\lambda$ is an eigenvalue of $u_P$.
Moreover $u_P \neq \id_P \lambda$, so $L:=\im(u_P-\id_P \lambda)$ is a $1$-dimensional $\calQ$-linear subspace of $P$.
As it is invariant under $u_P$, every nonzero vector of it is an eigenvector of $u_P$.

We shall prove that $x' \in \im(u_P-\id_P \lambda)$. To see this, note that if $\lambda$ is the sole eigenvalue of $u_P$,
then $\im(u_P-\id_P \lambda)=\Ker(u_P-\id_P \lambda)$, and $x' \in \Ker(u_P-\id_P \lambda)$ as we know from the start that $x'$
is an eigenvector of $u_P$.

Assume now that $\lambda$ is not the sole eigenvalue of $u_P$. Then $\im(u_P-\id_P \lambda)\oplus \Ker(u_P-\id_P \lambda)=P$
and it is not difficult to see that every eigenvector of $u_P$ belongs to $\im(u_P-\id_P \lambda)$ or $\Ker(u_P-\id_P \lambda)$:
indeed, $\lambda$ differs from any eigenvalue attached to an eigenvector in $\im(u_P-\id_P \lambda)$,
and hence is not conjugated to any such eigenvalue (because $\lambda$ is central).
Assume that $x' \not\in L$: then $x'+L$ contains no vector of $L$ and hence must be included in $\Ker(u_P-\id_P \lambda)$
(remember from the previous step that all the nonzero elements of $x'+L$ are eigenvectors of $u_P$); taking the translation vector spaces,
this leads to $L \subseteq \Ker(u_P-\id_P \lambda)$ and then to $L=\Ker(u_P-\id_P \lambda)$ by comparing dimensions, thereby contradicting
the assumption that $L \oplus \Ker(u_P-\id_P \lambda)=P$.

Hence in any case we have shown that $x' \in L$.
Remembering that $V=P \oplus z\calQ$ and that $u-\id_V \lambda$ is $\calQ$-linear, it follows that
$u(x)-x\lambda \in L$ for all $x \in V$, and we conclude that $u=\id_V \lambda+a$ for some rank $1$ endomorphism $a$ of $V$
(with range $L$). Hence $u$ is of special type II.

\vskip 3mm
\noindent \textbf{Case 2.} $q \not\in \F$. \\
Assume that $u_P$ is not diagonalisable and unispectral. Then the eigenvectors of $u_P$ for the eigenvalue $q$ all belong to a $1$-dimensional
$\calQ$-linear subspace $L$.
Moreover, either $u_P$ has only one invariant $1$-dimensional $\calQ$-linear subspace, or it has exactly two.
In any case, the set of all eigenvectors of $u_P$ is included in the union of two $4$-dimensional $\F$-linear subspaces of $P$,
whereas $\{u(y)-yq+x' \mid y \in P\}$ is an $\F$-affine subspace with dimension $6$.
The latter cannot be covered by two $4$-dimensional $\F$-affine subspaces, and hence we obtain a contradiction.

It follows that $u_P$ is diagonalisable and unispectral.
We can find a vector $x \in P \setminus \{0\}$ and some $\delta \in \calQ^\times$ such that $x'=x\delta$
and $u(x)=xq$, and then we extend $x$ into a basis $(x,y)$ of $P$ such that $u(y)=yq$.
Recall that $u(z)=zq+x'$.
Choose $\beta \in \calQ$ that does not commute with $q$.
Then the vector $u(y\beta)-(y\beta)q+x' $
is either zero or an eigenvector of $u_P$. Yet this vector is simply $x \delta +y [q,\beta]$, so it is nonzero
and hence it is an eigenvector of $u$, with eigenvalue denoted by $q'$. We deduce that $q\delta=\delta q'$ and $q [q,\beta]= [q,\beta] q'$.
Thus, both $[q,\beta]$ and $\delta$ conjugate $q'$ into $q$, to the effect that $\alpha:=[q,\beta] \delta^{-1} $ commutes with $q$, and hence
$\delta=\alpha^{-1}[q,\beta]=[q,\gamma]$ for $\gamma:=\alpha^{-1}\beta \in \calQ$.
It follows that $z':=z-x \gamma$ satisfies
$$u(z')=zq+x \delta-x q \gamma=zq-x \gamma q=z' q,$$
and we conclude that $(x,y,z')$ is a basis of $V$ that consists of eigenvectors of $u$ for the eigenvalue $q$.
Hence $u$ is of special type III.
\end{proof}

\section{Inductive proof of the main theorem}\label{section:induction}

We are now in the position to prove the difficult implication in Theorem \ref{theo:n>=3}, which is (i) $\Rightarrow$ (ii). The proof will be by induction on $n$.
Throughout, we fix an arbitrary quaternion division algebra $\calQ$ over a field $\F$.

A key idea is the following classical (and obvious) lemma:

\begin{lemma}\label{lemma:extension}
Let $R$ be a ring, and let $M \in \Mat_n(R)$ be of the form
$\begin{bmatrix}
0 & [?]_{1 \times (n-1)} \\
[?]_{(n-1) \times 1} & N
\end{bmatrix}$ where $N \in \Mat_{n-1}(R)$ is similar to matrix with diagonal zero. Then $M$ is similar to a matrix with diagonal zero.
\end{lemma}

\subsection{The case $n=3$}\label{section:n=3}

Let $V$ be a $3$-dimensional vector space over $\calQ$, and let $u \in \End(V)$ be such that $\tr(u)=0$
but $u$ is of none of special types I to III.
We wish to prove that $u$ is represented by a matrix with diagonal zero.
By Lemma \ref{lemma:3LLD}, there is a vector $x \in V$ such that $(x,u(x),u^2(x))$ is a basis of $V$.
We write the matrix of $u$ in the basis $(x,u(x),u^2(x))$ as
$$M=\begin{bmatrix}
0 & 0 & c \\
1 & 0 & b \\
0 & 1 & a
\end{bmatrix} \quad \text{where $(a,b,c)\in \calQ^3$.}$$
Let $\delta \in \calQ$, and consider the basis $(x,u(x),u^2(x)+x \delta)$ instead:
we obtain that the new representing matrix for $u$ equals
$$M(\delta):=\begin{bmatrix}
0 & -\delta & c-a \delta \\
1 & 0 & b+\delta \\
0 & 1 & a
\end{bmatrix}.$$
Now, since $t(a)=\tr(M)=\tr(u)=0$, we gather from Lemma \ref{lemma:completionn=2} that $\delta$ can be chosen so that
$\begin{bmatrix}
0 & b+\delta \\
1 & a
\end{bmatrix}$ is the sum of two nilpotent matrices, and hence similar to a matrix with diagonal zero (by Theorem \ref{theo:n=2}).
By Lemma \ref{lemma:extension} we infer that $M(\delta)$ is similar to a matrix with diagonal zero,
and hence $u$ is represented by a matrix with diagonal zero.

\subsection{The case $n \geq 4$}

Now, we proceed by induction. Let $n \geq 4$, and assume that the implication (i) $\Rightarrow$ (ii) in Theorem \ref{theo:n>=3}
has been proved for all $(n-1)$-dimensional vector spaces over $\calQ$.
Let $V$ be an $n$-dimensional vector space over $\calQ$, and $u$ be an endomorphism of it with trace zero.
We will assume that $u$ is not represented by a matrix with diagonal zero and is not of special type I, and we shall prove that
$u$ must then be of special type II.

Since $u$ is not of special type I, we use Lemma \ref{lemma:carachomotheties} to retrieve
a vector $x \in V$ such that $x$ and $u(x)$ are linearly independent.
We extend $(x,u(x))$ into a basis $(e_1,\dots,e_n)$ of $V$ (with $e_1=x$ and $e_2=u(x)$).
We write the matrix of $u$ in that basis as
$$M=\begin{bmatrix}
0 & [?]_{1 \times (n-1)} \\
C & N
\end{bmatrix}$$
with $N \in \Mat_{n-1}(\calQ)$ and $C=\begin{bmatrix}
1 & 0 & 0 & \cdots & 0
\end{bmatrix}^T$.
Now, as in Section \ref{section:n=3}, the idea is to replace the basis with
$(e_1,e_2,e_3+e_1q_3,e_4+e_1q_4,\dots,e_n+e_1q_n)$ for an arbitrary list $(q_3,\dots,q_n)$ of elements of $\calQ$.
In the resulting representing matrix $M'$ of $u$, this affects the first row only in the last $n-1$ entries, and this modifies
the lower block $N$ by adding the matrix
$$L(q_3,\dots,q_n)=\begin{bmatrix}
0 & q_3 & \cdots & q_n \\
[0]_{(n-2) \times 1} & [0]_{(n-2) \times 1} & \cdots & [0]_{(n-2) \times 1}
\end{bmatrix}$$
to it.
It follows that, whatever the choice of the $q_i$'s, the resulting matrix
$N+L(q_3,\dots,q_n)$ is not similar to a matrix with diagonal zero. Yet in any case this matrix has trace zero,
so by induction it must be of one of special types I to III (i.e., it represents endomorphisms of such types).

At this point, it is necessary to split the discussion into two cases. We will start by considering the case where $N$
has special type III. The following lemma will be useful in both cases.

\begin{lemma}\label{lemma:compatlemma1}
Let $v$ be an endomorphism of a finite-dimensional vector space $W$ over $\calQ$, with special type I or II. Let $b$ be a rank $1$ endomorphism of $W$.
Then $v+b$ is of special type III only if it is also of special type I.
\end{lemma}

\begin{proof}
Assume that $v+b$ is of special type III, so that $W$ has dimension $3$.
Write $v=\id_V \lambda+a$ for some $a \in \End(V)$ of rank at most $1$, and some $\lambda \in \F$.
Since $\dim W=3$ there is a nonzero vector $x$ in $\Ker(a) \cap \Ker (b)$.
Then $(v+b)(x)=x\lambda$, so $\lambda$ is an eigenvalue of $v+b$. Since $v+b$ is of special type III, we conclude that $v+b=\id_W \lambda$ with $\lambda \neq 0$,
and hence $v+b$ is of special type I.
\end{proof}

\subsection{Case 1: $N$ is not of special type I or II}

Then $n=4$ and $N$ is of special type III.
Consider an arbitrary nonzero pair $(q_3,q_4) \in \calQ^2 \setminus \{(0,0)\}$ and note that $L(q_3,q_4)$ has rank at most $1$.
As seen earlier, $N+L(q_3,q_4)$ must be of special type I to III. If it were of special type I or II, then Lemma \ref{lemma:compatlemma1}
would yield that $N=(N+L(q_3,q_4))-L(q_3,q_4)$, which is of special type III, is actually of special type I, which would contradict our assumptions.
Hence $N+L(q_3,q_4)$ is of special type III.

We shall now prove that this leads to a contradiction. This is proved as a separate lemma, in which we consider the dual problem.

\begin{lemma}\label{lemma:translatetypeIII}
Let $v$ be an endomorphism of special type III, but not of special type I, of a vector space $W$ over $\calQ$, and let $P$ be a $2$-dimensional linear subspace of $W$.
Then there exists a nilpotent endomorphism $w$ of $W$ with kernel $P$ and such that $v+w$ is not of special type III.
\end{lemma}

\begin{proof}
We assume on the contrary that $v+w$ is of special type III whatever the nilpotent endomorphism $w \in \End(W)$ with kernel $P$, and we seek a contradiction.
First of all, we choose an eigenvalue $q$ of $v$. Note that $q \not\in \F$ because $v$ is not of special type I. We shall prove that $P$ contains an eigenvector of $v$ for the eigenvalue $q$.
To prove this, we choose a basis $(f_1,f_2,f_3)$ of $V$ in which all the vectors are eigenvectors for the eigenvalue $q$
(this is possible). Since $q \not\in \F$, $\F[q]$ is a quadratic extension of $\F$, the $\F[q]$-linear subspace
$G:=f_1\F[q]+f_2\F[q]+f_3\F[q]$ has dimension $3$ over $\F[q]$, whereas $W$ has dimension $6$ over $\F[q]$ and
$P$ has dimension $4$ over $\F[q]$. It follows by Grassmann's formula that $G \cap P \neq \{0\}$, which yields our first
point because every nonzero vector of $G$ is an eigenvector of $v$ for the eigenvalue $q$.

Now, we choose $e_1 \in P \setminus \{0\}$ such that $v(e_1)=e_1 q$, we extend it into a basis $(e_1,e_2)$ of the $\calQ$-vector space $P$, and finally into a basis
$(e_1,e_2,e_3)$ of the $\calQ$-vector space $V$. The matrix of $v$ in that basis now looks as
$$A=\begin{bmatrix}
q & ? & ? \\
0 & ? & ? \\
0 & ? & c
\end{bmatrix} \quad \text{where $c \in \calQ$.}$$
The nilpotent endomorphisms of $W$ with kernel $P$ are exactly the endomorphisms that are represented by a matrix of the form
$$\begin{bmatrix}
0 & 0 & \alpha \\
0 & 0 & \beta \\
0 & 0 & 0
\end{bmatrix} \quad \text{with $(\alpha,\beta)\in \calQ^2 \setminus \{(0,0)\}$.}$$
Then, by adding to $v$ an appropriate endomorphism of this kind (if necessary) we obtain an endomorphism of special type III that is represented by a matrix of the form
$$\begin{bmatrix}
q & ? & 0 \\
0 & ? & 0 \\
0 & ? & c
\end{bmatrix}.$$
Clearly $q$ and $c$ are eigenvalues of the latter, so they must be conjugates of one another.
Hence, by replacing $e_3$ with $e_3 \gamma$ for a well-chosen $\gamma \in \calQ^\times$ we can reduce the situation to the one where
$c=q$.

Next, choose $\delta \in \calQ$ such that $t(\delta) \neq 0$. The starting assumptions yield a matrix of the form
$$\begin{bmatrix}
q & ? & \delta \\
0 & ? & 0 \\
0 & ? & q
\end{bmatrix}$$
that represents an endomorphism of special type III in the basis $(e_1,e_2,e_3)$.
We note that $\Vect(e_1,e_3)$ is invariant under such an endomorphism, and we deduce from
Proposition \ref{prop:induceddiagonalisable} that
$\begin{bmatrix}
q & \delta \\
0 & q
\end{bmatrix}$ is diagonalisable. Then Lemma \ref{lemma:diago2by2} yields that $\delta=[q,d]$ for some $d \in \calQ$.
Yet this yields $t(\delta)=0$, contradicting our assumption.
This final contradiction completes the proof.
\end{proof}

To complete the proof in the present case, we consider the transconjugation $M \in \Mat_3(\calQ) \mapsto (\overline{m_{j,i}})_{1 \leq i,j \leq 3}$,
and we note that it takes diagonal matrices to diagonal matrices and preserves similarity
(this is because $(BA)^\star=A^\star B^\star$, and it is crucial to use the conjugation and not simply a transposition here).
In particular the transconjugation preserves the property of being of special type III.
Now, applying this we get that for all $(q_3,q_4) \in \calQ^2$ the matrix
$N^\star+\begin{bmatrix}
0 & 0 & 0 \\
\overline{q_3} & 0 & 0 \\
\overline{q_4} & 0 & 0
\end{bmatrix}$ is of special type III. This amounts to saying that for the endomorphism $v$
of $\calQ^3$ that is represented in the standard basis by $N^\star$, the endomorphism $v$ is of special type III, and $v+w$ is of special type III for every nilpotent endomorphism $w$ of $\calQ^3$ with kernel $\{0\} \times \calQ^2$.
Yet this runs in contradiction with Lemma \ref{lemma:translatetypeIII}. Hence the assumed case is impossible, and $N$ must actually be of special type I or II.

\subsection{Case 2: $N$ is of special type I or II}

Previously we have proved that $N$ is of special type I or II.
We proceed to prove that $u$ is of special type II.

To start with, $N=I_{n-1} \lambda+A$ for some $\lambda \in \F$ and some $A \in \Mat_{n-1}(\calQ)$ with rank at most $1$.
Consider an arbitrary nonzero list $(q_3,\dots,q_n) \in \calQ^{n-2}$ and note that $L(q_3,\dots,q_n)$ has rank at most $1$.
By Lemma \ref{lemma:compatlemma1}, $N+L(q_3,\dots,q_n)$ is of special type III only if it is of special type I, so in any case it is of special type I or II, yielding that
$I_{n-1}\lambda+A+L(q_3,\dots,q_n)=I_{n-1} \mu+B$ for some $B \in \Mat_{n-1}(\calQ)$ with rank at most $1$, and some $\mu \in \F$.
This yields
$$I_{n-1}(\mu-\lambda)=A-B+L(q_3,\dots,q_n).$$

If $\lambda \neq \mu$, then $A-B=I_{n-1}(\mu-\lambda)-L(q_3,\dots,q_n)$
is invertible (being upper-triangular with nonzero diagonal entries), yet $\rk(A-B) \leq \rk(A)+\rk(B) \leq 2$, and we contradict the fact that $n-1 \geq 3$.

Therefore $\lambda=\mu$. In turn this shows that $A+L(q_3,\dots,q_n)$ has rank at most $1$.
Now, we claim that $\im A \subseteq \calQ \times \{0_{n-2}\}$.
Assume that this is false. Then $\rk A=1$.
Then from Lemma \ref{lemma:sumoftworank1} we see that for every nonzero list $(q_3,\dots,q_n) \in \calQ^{n-2}$, as $A+L(q_3,\dots,q_n)$ has rank at most $1$
and $\im A \neq \im(L(q_3,\dots,q_n))$, we find that $\Ker A=\Ker (L(q_3,\dots,q_n))$.
Since $n>3$ we obtain a contradiction by varying the list $(q_3,\dots,q_n)$.

We conclude that $\im A \subseteq \calQ \times \{0_{n-2}\}$, i.e., $N-\lambda I_{n-1}$ has all its rows zero starting from the second one, and
hence $M-\lambda I_n$ has all its rows zero starting from the third one.

Focusing only on the first two columns, we generalize part of these findings as follows:

\begin{claim}\label{lastclaim}
For every vector $y \in V \setminus \{0\}$ that is not an eigenvector of $u$, one has
$u^2(y)\in \Vect(y,u(y))$.
\end{claim}

We will conclude that this leads to $u$ being of special type II.
To see this, we consider the vector space endomorphism $v:=u-\id_V\lambda$. Let $k \in \lcro 3,n\rcro$.
We note that the $\calQ$-linear subspace $\Vect(e_1,e_2,e_k)$ is invariant under $v$, and the matrix of the induced endomorphism $v_k$ takes the form
$$\begin{bmatrix}
-\lambda & m_{1,2} & m_{1,k} \\
1 & m_{2,2} & m_{2,k} \\
0 & 0 & 0
\end{bmatrix},$$
where $m_{1,j}$ and $m_{2,j}$ denote the first two entries in the $j$-th column of $M$.

By Claim \ref{lastclaim}, we can apply Lemma \ref{lemma:3LLD} to $v_k$
and recover that it is of special type I, II or III. Yet it is obvious from the first column that $v_k$ is not of special type I, and it is also not of special type III
since it is not an isomorphism (judging from the third row of the above matrix). Hence $v_k$ is of special type II.
We infer that $v_k=\id \mu+a_k$ for some $\mu \in \F$ and some rank $1$ endomorphism $a_k$ of $\Vect(e_1,e_2,e_k)$.
As seen from the first column of the above matrix we successively find that $a_k(e_1) \neq 0$ and $\im a_k \subset \Vect(e_1,e_2)$.
Hence the third diagonal entry of the above matrix equals $\mu$, to the effect that $\mu=0$, whence $v_k$ has rank $1$.
In particular, the second and third columns of the above matrix are multiples of the first one on the right with elements of $\calQ$.
Varying $k$, we conclude that all the columns of the matrix of $v$ in $(e_1,\dots,e_n)$ are multiples of the first one on the right with elements of $\calQ$.
Hence $\rk(v)=1$. As $u=\id_V \lambda+v$, we conclude that $u$ is of special type II, which completes the inductive step.

Therefore, implication (i) $\Rightarrow$ (ii) in Theorem \ref{theo:n>=3} is now proved. The remaining implications have already been explained, and our study is complete.

\end{document}